\documentclass[a4paper,10pt]{article}

\usepackage{color}
\usepackage[english]{babel}
\usepackage{amsmath,amsfonts,amssymb,amsthm}
\usepackage{url}
\usepackage{graphicx}

\newtheorem{theorem}{Theorem}[section]

\newtheorem{lemma}{Lemma}[section]
\newtheorem{proposition}{Proposition}[section]
\newtheorem{corollary}{Corollary}[section]

\theoremstyle{definition}
\newtheorem{remark}{Remark}[section]

\numberwithin{equation}{section}

\newcommand\blfootnote[1]{\begingroup\renewcommand\thefootnote{}\footnote{#1}\addtocounter{footnote}{-1}\endgroup}

\begin{document}

\title{
{\bf\Large Monotone wave fronts for $(p, q)$-Laplacian driven reaction-diffusion equations }}

\author{
\vspace{1mm}
\\
\vspace{1mm}\\
{\bf\large Maurizio Garrione, \, Marta Strani}
\vspace{1mm}\\
{\it\small Dipartimento di Matematica e Applicazioni, Universit\`a di Milano-Bicocca}\\
{\it\small via Cozzi 55}, {\it\small 20125 Milano, Italy}\\
{\it\small e-mail addresses: maurizio.garrione@unimib.it, \, marta.strani@unimib.it}\vspace{1mm}\\}

\date{}

\maketitle

\vspace{-2mm}

\begin{abstract}
We study the existence of monotone heteroclinic traveling waves for the $1$-dimensional reaction-diffusion equation 
$$
u_t = (\vert u_x \vert^{p-2} u_x + \vert u_x \vert^{q-2} u_x)_x + f(u),
$$
where the non-homogeneous operator appearing on the right-hand side is known as $(p, q)$-Laplacian. Here we assume that $2 \leq q < p$ and $f$ is a nonlinearity of Fisher type, namely it is always positive out of its zeros. We give an estimate of the critical speed and we comment on the roles of $p$ and $q$ in the dynamics, providing some numerical simulations.
\end{abstract}


\section{Introduction}

This paper deals with wave fronts for the $1$-dimensional reaction-diffusion equation
\begin{equation}\label{eq0intro}
u_t = (\vert u_x \vert^{p-2} u_x + \vert u_x \vert^{q-2} u_x)_x + f(u), 
\end{equation}
where $2 \leq q < p$ and $f$ is a regular function such that $f(0)=0=f(1)$. The operator driving the diffusion is a $(p, q)$-Laplacian type one, and it is characterized by the presence of two (cooperative) terms, the former prevailing for large gradients and the latter prevailing for small ones. 
\smallbreak
\noindent
In these last years, there has been an increasing interest for problems driven by this operator, especially in the stationary (elliptic) case, both in bounded and unbounded domains (see, e.g., \cite{BarCanSal16, YanPer16} and the references therein). A $(p, q)$-Laplacian type operator may arise quite naturally, for instance, when taking into account more terms in the Taylor expansion of a quasi-linear operator, and finds application, e.g, in modeling chemical reactions and some physical phenomena (see for instance \cite{BenForPis98, CheIly05, HeLi08}). Of course, equations like \eqref{eq0intro} have been widely studied in the case $p=q$, i.e. for the $p$-Laplacian operator (we limit ourselves to mention some of the papers which are strictly related to the topic of the present investigation, namely \cite{AudVazPP, CoeSan14, EngGavSan13, FeiHilPetTak14, GavSan15}). The main difficulty in the case $p \neq q$ comes of course from the lack of homogeneity of the operator, which turns into preventing a straight application of almost any usual argument and requires some supplementary care.
\smallbreak
\noindent
By the expression \emph{wave front}, we mean a heteroclinic solution of the type $u(t, x)=v(x+ct)$ connecting two equilibria $u_-, u_+$ (namely, such that $v(-\infty)=u_-$, $v(+\infty)=u_+$). In view of the assumptions on $f$, for \eqref{eq0intro} such profiles will connect the two equilibria $u_-=0$ and $u_+=1$. Dealing with reaction-diffusion equations, the search for these solutions is the first natural step to get a better insight into the dynamics of the model. They play the role of phase transitions between two steady states of the system and, in the context of the original genetics models proposed by Fisher \cite{Fis37, KolPetPis37}, they actually were meant to describe the relative propagation of a favourable gene inside the considered population. 
For this reason, it is very natural to assume that $0 \leq u \leq 1$.  
\smallbreak
\noindent
Writing the equation for the traveling fronts by setting $u(t, x)=v(x+ct)$, equation \eqref{eq0intro} leads to the problem 
\begin{equation}\label{problema}
\left\{
\begin{array}{l}
(\vert v' \vert^{p-2} v' + \vert v' \vert^{q-2} v')' -cv' + f(v) = 0, \qquad (v=v(z)) \\
v(-\infty)=0, \; v(+\infty)=1.
\end{array}
\right.
\end{equation}
We will limit ourselves to the case when $f$ is strictly positive out of the equilibria (a Fisher-type case).  
To better compare with the existing literature, moreover, we will search since the very beginning for \emph{monotone} fronts, by-passing some problems which may arise from the fact that uniqueness is not ensured for the differential equation in \eqref{problema}.
\smallbreak
\noindent
The first question we are interested in regards the possible speeds of propagation of the traveling waves, namely the values $c \in \mathbb{R}$ for which \eqref{problema} has a monotone solution, which are called
\emph{admissible speeds}. In the classical case $p=q=2$, this problem has been extensively studied (see, e.g., the references in \cite{BonSan06, GarSan15}) and it has been shown that the set of the admissible speeds is an unbounded interval, whose lower endpoint is called \emph{critical speed} and is denoted by $c^*$. The value of $c^*$ depends on the behavior of the reaction term $f$: if $f$ is always below its tangent line in $0$, it is $c^*=2\sqrt{f'(0)}$. The existence of a critical speed $c^*$ has been extended to the general case $p=q$ in \cite{CoeSan14, EngGavSan13, GavSan15}, showing moreover that $c^*$ admits a variational characterization like in the linear case. 
\smallbreak
\noindent
The purpose of the present investigation is to analyze when and how the estimates for the critical speed change in presence of the two different powers $p$ and $q$ in the expression of the operator. 
Indeed, the regimes acting for $\vert z \vert$ large involve very small gradients, since for a monotone heteroclinic solution it is $v' \to 0$ at infinity. Hence, the dominant term for $z=x+ct$ large will be the one with exponent $q$. However, no information is a priori available on the derivative $v'$, so that it may be that there are regions where the term with the power $p$ is prevailing. Intuitively, such effect should be more accentuated the more there is room for the derivative to grow (and this is well seen in Figure \ref{B} in Section \ref{sez3}). Indeed, it turns out that the value of the constant providing a bound for the critical speed is actually influenced by this interaction when the nonlinearity is sufficiently large, see Theorem \ref{mainT} below. On the other hand, in the case $p=q$, we recover the results proved in \cite{CoeSan14, EngGavSan13} for the $p$-Laplacian operator. 
\smallbreak
\noindent
After our main statement, in Section \ref{osservazioni} we give some complementary remarks mentioning, in particular, the case when the two diffusive terms are in competition, namely considering the equation
$$
(\vert v' \vert^{q-2} v' - \vert v' \vert^{p-2} v')' -cv' + f(v) = 0.
$$
Avoiding a degeneration of the problem which may appear if the derivative reaches the value $1$, we also briefly discuss the picture for fronts in this situation (see Remark \ref{conilmeno}). 
\smallbreak
\noindent
The technique exploited makes use of a change of variables allowing to lead the original problem back to a two-point first order problem on $[0, 1]$, to be dealt with through a shooting technique.
\smallbreak
\noindent
The paper ends with Section \ref{sez3}, devoted to some numerical simulations illustrating our main results.

\section{Main result and discussion}\label{sez2}

In this section, we study the admissible speeds for the $1$-dimensional partial
differential equation
\begin{equation}\label{NLFB}
u_t=(\vert u_x \vert^{p-2} u_x + \vert u_x \vert^{q-2} u_x) + f(u),
\end{equation}
assuming that $2 \leq q < p$ and that the reaction term $f(u)$ satisfies the following hypothesis:
\smallbreak
\noindent
\textbf{(f)} \quad $f: [0, 1] \to \mathbb{R}$ is a $C^1$-function such that $f(0)=0=f(1)$ and $f(u) > 0$ for $u \in \,]0, 1[\,$.

\subsection{Statement and proof of the main result}

As discussed above, the search for monotone heteroclinic traveling waves $u(t, x) = v(x + ct)$ solving \eqref{NLFB} leads us to study the second order boundary value problem 
\begin{equation}\label{IIordE}
\left\{
\begin{array}{l}
(\vert v'\vert^{p-2} v' + \vert v' \vert^{q-2} v')' - c v' + f(v) = 0 \\
v(-\infty)=0, \; v(+\infty)=1, \; v' > 0.
\end{array}
\right.
\end{equation}
Since $v$ is strictly monotone, the map $z \mapsto v(z)$ is invertible, so that we can take $v$ as the new independent variable and write $z=z(v)$. Setting $\phi(v)=v'(z(v))$ and proceeding similarly, e.g., as in \cite{EngGavSan13, GarSan15}, the differential equation in \eqref{IIordE} is rewritten as  
$$
\frac{d}{dv} Q\left(\phi(v)\right)-c \phi(v)+f(v)=0,
$$
being $Q(s)$ the primitive of $(p-1)\vert s \vert^{p-2} s + (q-1) \vert s \vert^{q-2} s$ that satisfies $Q(0)=0$.  In particular, it turns out that 
$$
Q(s)= \frac{p-1}{p} \vert s \vert^p + \frac{q-1}{q} \vert s \vert^q
$$
and $Q(\cdot)$ is invertible for $s \geq 0$; it is worth noticing that we can actually invert $Q(\phi(v))$ since we search for \emph{increasing} fronts, such that $\phi(v) > 0$.
Denoting by $R(\cdot)$ the (functional) inverse of $Q(\cdot)$ and setting $y(v)=Q(\phi(v))$,  
problem \eqref{IIordE} is thus rephrased as 
\begin{equation}\label{generale}
\left\{
\begin{array}{l}
\displaystyle y'= c R(y) - f(v),\\
\\
y(0)=0=y(1), \; y(v) > 0 \textrm{ for } v \in \,]0, 1[\,.
\end{array}
\right.
\end{equation}
Moreover, since
$$
\lim_{s\to 0^+} \frac{Q(s)}{s^q}= \frac{q-1}{q} \quad {\rm and} \quad \lim_{s\to \infty} \frac{Q(s)}{s^p}= \frac{p-1}{p},
$$
it holds
\begin{equation}\label{BdR}
\lim_{s\to 0^+} \frac{R(s)}{s^{1/q}}= \left(\frac{q}{q-1}\right)^{\frac{1}{q}} \quad {\rm and} \quad \lim_{s\to \infty} \frac{R(s)}{s^{1/p}}= \left(\frac{p}{p-1}\right)^{\frac{1}{p}}.
\end{equation}
We now state our main result estimating the \emph{critical} speed for problem \eqref{IIordE}; henceforth, for a number $r > 1$, we will denote by $r'$ its conjugate exponent, namely the positive number such that $1/r + 1/r' = 1$. 
\begin{theorem}\label{mainT}
Let $f: [0, 1] \to \mathbb{R}$ satisfy assumption \textbf{(f)}. Moreover, let
\begin{equation*}
\lim_{s \to 0^+} \frac{f(s)}{s^{q'-1}} = L_0 \quad {\rm and} \quad \sup_{s \in \,]0, 1]} \frac{f(s)}{s^{q'-1}} = L_+ < +\infty
\end{equation*}
and assume that there exists $k > 0$ such that 
\begin{equation}\label{hyp1}
f(s) \leq k(1-s), \quad \textrm{ for every } s \in [0, 1].
\end{equation}
Then, there exists $c^*$ satisfying
\begin{equation}\label{stima}
L_0^{1/q'} q'^{1/q'} q^{1/q} \leq c^* \leq c_+,  
\end{equation}
with $c_+$ given by 
\begin{equation}\label{c+}
c_+= \left\{
\begin{array}{ll}
\displaystyle \frac{L_+^{\frac{1}{q'}} \, q}{q-1} \cdot (p+q-2)^{\frac{1}{q}} & \textrm{ if } \, L_+ \leq p+q-2 \vspace{0.3cm} \\
\displaystyle \frac{p (p+q-2)}{q-1} & \textrm{ if } \, \displaystyle p+q-2 < L_+ \leq \frac{p-1}{q-1} (p+q-2)
\vspace{0.3cm} \\
\displaystyle \frac{L_+^{\frac{1}{p'}} \, p}{(q-1)^{\frac{1}{p}} (p-1)^{\frac{p-1}{p}}} \; (p+q-2)^{\frac{1}{p}} & \textrm{ if } \, \displaystyle L_+ > \frac{p-1}{q-1} (p+q-2),
\end{array}
\right.
\end{equation}
such that every $c \geq c^*$ is admissible for problem \eqref{IIordE}. 
\end{theorem}
We first observe that assumption \eqref{hyp1} is not used to solve problem \eqref{generale} but ensures that the solutions of \eqref{generale} are in a one-to-one correspondence with global strictly monotone profiles, defined on the whole $\mathbb{R}$ and solving \eqref{IIordE}. Such profiles are reconstructed, up to translations, by solving the problem 
\begin{equation}\label{reconstr}
\frac{du}{dt}= R(y(u(t))), \quad u(0)=1/2; 
\end{equation}
similarly as in \cite[Proposition 2.3]{CoeSan14}, in view of \eqref{BdR} we have that the solution of $y'=cR(y) - f(u)$ has a growth controlled by $u^{q'}$ near $0$, while \eqref{hyp1} ensures that, for the same solution, it holds $y(u) \leq C(1-u)^2$. It follows that the solution of \eqref{reconstr} is defined (and strictly monotone) on the whole $\mathbb{R}$ (see \cite[Proposition 2.3]{CoeSan14} for further details).  
\smallbreak
\noindent 
In order to prove Theorem \ref{mainT}, we first give a lower bound on the critical speed $c^*$. This will only come from the behavior of $R(y)$ near $0$, which, in view of \eqref{BdR}, will be ruled completely by the \emph{smallest} power inside the operator. It is indeed sufficient to assume that  
\begin{equation}\label{stimaf0}
\lim_{s \to 0^+} \frac{f(s)}{s^{q'-1}} = L_0,
\end{equation}
for a positive constant $L_0$, 
in order to state the following result, which follows similarly as in \cite{CoeSan14}. 
\begin{lemma}\label{upperbd}
Assume \eqref{stimaf0} and let $y(v)$ be a solution of the differential equation in \eqref{generale}, satisfying $y(0)=0$ and $y(v) > 0$ in a right neighborhood of $0$.
Then,
\begin{equation}\label{bound1}
c \geq L_0^{1/q'} q'^{1/q'} q^{1/q}.
\end{equation}
\end{lemma}
\begin{proof}
As usual, we want to find an equation for $l=\limsup_{v \to 0^+ } \frac{y(v)}{v^{q'}} \geq 0$.
From the differential equation in \eqref{generale}, dividing by $q' v^{q'-1}$ and using De l'Hopital rule, we obtain 
$$
l = \limsup_{v \to 0^+}\frac{y(v)}{v^{q'}} \leq 
 c\limsup_{v \to 0^+} \frac{R(y(v))}{q' v^{q'-1}} - \frac{L_0}{q'},
$$
from which, recalling \eqref{BdR}, it holds
$$
l - \frac{c}{q'^{1/q'}} l^{1/q} + \frac{L_0}{q'} \leq 0.
$$
For this inequality to be true, at least it has to be 
\begin{equation}\label{intermedia}
\frac{L_0}{q'} \leq \max_{x \in \mathbb{R}} \frac{c}{q'^{1/q'}} x^{1/q} - x;
\end{equation}
the maximum at the right-hand side is attained at $x=c^{q'}/(q' q^{q'})$ and a direct computation shows that \eqref{intermedia} can be true only if $c$ satisfies \eqref{bound1}. 
\end{proof}
Our second result provides the upper bound for the admissible speeds. As usual in performing this estimate, a global bound on $f$ is needed. We thus assume that 
\begin{equation}\label{stimaf}
\sup_{s \in \,]0, 1]} \frac{f(s)}{s^{q'-1}} = L_+ < +\infty.
\end{equation}
The strategy is the following shooting-type one. First, we observe that, for every $c$, $y \equiv 0$ is a subsolution to the backward Cauchy problem 
\begin{equation}\label{cback}
\left\{
\begin{array}{l}
\displaystyle y'= c R(y) - f(v),\\
\\
y(1)=0, \;
\end{array}
\right.
\end{equation}
since $R(0)=0$ and $f$ has positive sign. Hence, the solution $y(v)$ of \eqref{cback} (uniqueness holds for \eqref{cback} since $R$ is increasing) is nonnegative. Even more, $y(v)$ never vanishes in $\,]0, 1[\,$, due to the sign of $f$, so it is well defined up to $v=0$. If $y(v)=0$, then we are done; otherwise, the idea is to search for suitable values of $c$ for which there exists a positive subsolution for the forward Cauchy problem  
\begin{equation}\label{cfw}
\left\{
\begin{array}{l}
\displaystyle y'= c R(y) - f(v),\\
\\
y(0)=0, \;
\end{array}
\right.
\end{equation} 
defined on the whole $[0, 1]$. In this way, a standard uniqueness argument (out of $y=0$) will ensure that, for these values of $c$, we can find the desired connection between $0$ and $1$. 
\smallbreak
\noindent
We have the following result.
\begin{proposition}\label{teo1}
Let $f$ satisfy assumption \eqref{stimaf}. Then, for every $c \geq c_+$, where $c_+$ is given by \eqref{c+}, there exists a positive subsolution of \eqref{cfw}. 
\end{proposition}
\begin{proof}
We will search the subsolution of \eqref{cfw} in the form
\begin{equation}\label{sottosol}
 y(u)= Q(\beta u^{\alpha}), \quad u \in [0, 1],
\end{equation}
with $\alpha>0$ and $\beta>0$ to be determined. By substituting into the differential equation, noticing that $y(u) > 0$ for every $u \in \,]0, 1]$ so that $Q(y(u))$ is always invertible in $[0, 1]$, in order for $y(u)$ to be a subsolution it has to be
\begin{equation}\label{1}
\alpha\left[ (p-1)\beta^p u^{\alpha p-1}+ (q-1)\beta^q u^{\alpha q-1}\right]-c \beta u^\alpha + f(u) \leq 0.
\end{equation}
If we now divide by $u^{\alpha}$, it is easy to check that the smallest positive $\alpha$ such that both $\alpha p-1-\alpha$ and $\alpha q-1-\alpha$ are positive is $\alpha=q'-1$. Hence we set $\alpha=q'-1$ and \eqref{1} reads as
$$
(q'-1)\left[ (p-1)\beta^p u^{p(q'-1)-1}+ (q-1)\beta^q u^{q(q'-1)-1}\right]-c \beta u^{q'-1} + f(u) \leq 0,
$$
which, dividing by $u^{q'-1}$ and using \eqref{stimaf}, becomes
\begin{equation}\label{2}
(q'-1)\left[ (p-1)\beta^p u^{p(q'-1)-q'}+ (q-1)\beta^q \right]-c \beta  + L_+ \leq 0.
\end{equation}
Since \eqref{2} has to be satisfied for every $u \in [0, 1]$, we are thus led to investigate for which values of $c$ there exists $\beta > 0$ such that
\begin{equation}\label{3}
\mathcal{G}_c(\beta):=(q'-1)\left[ (p-1)\beta^p + (q-1)\beta^q \right]-c \beta  + L_+ = \frac{p-1}{q-1}\beta^p+ \beta^q-c\beta+L_+ \leq 0.
\end{equation}
Of course, if we are able to find $\bar{c}$ such that $\mathcal{G}_{\bar{c}}(\bar{\beta}) \leq 0$ for a suitable $\bar{\beta} > 0$, then the fact that $\mathcal{G}_c(\beta)$ is decreasing in the variable $c$ ensures that for every $c > \bar{c}$ it will be $\mathcal{G}_c(\bar{\beta}) \leq 0$, yielding the statement. 
We now briefly sketch the proof in the three cases highlighted in \eqref{c+}, which we refer to as \eqref{c+}$_{(i)}$, \eqref{c+}$_{(ii)}$ and \eqref{c+}$_{(iii)}$. 
\smallbreak
\noindent
\begin{itemize}
\item[1)]\underline{Case \eqref{c+}$_{(i)}$: $L_+ \leq p+q-2$.} In this case, we fix 
$$
\bar{c}=\frac{L_+^{\frac{q-1}{q}}\, q}{q-1} (p+q-2)^{\frac{1}{q}} 
$$ 
and show that $\min \mathcal{G}_{\bar{c}} \leq 0$. To this end, we consider the function 
$$
g_{\bar{c}}(\beta):=\left(\frac{p-1}{q-1}+1\right)\beta^q-\bar{c}\beta+L_+\leq 0, 
$$ 
which attains its minimum at the point 
$$
\bar \beta =\left(\frac{\bar{c}\,(q-1)}{q \, (p+q-2)}\right)^{\frac{1}{q-1}}.
$$
In view of the assumptions on $L_+$ and $\bar{c}$, we have that 
$$
\bar{\beta}= \left(\frac{L_+}{p+q-2}\right)^{\frac{1}{q}} \leq 1, \quad g_{\bar{c}}(\bar{\beta})=0. 
$$
On the other hand, since $\bar{\beta} \leq 1$, it is also 
$
\mathcal{G}_{\bar{c}}(\bar{\beta}) \leq g_{\bar{c}}(\bar{\beta}) = 0,
$
showing that the choice $c=\bar{c}$ allows to construct a positive lower solution of the form \eqref{sottosol} for $\beta=\bar{\beta}$. 
\item[2)]
\underline{Case \eqref{c+}$_{(ii)}$: $p+q-2 < L_+ \leq \frac{p-1}{q-1} (p+q-2)$.} 
Here we set $$\bar{c}= p(p+q-2)/(q-1)$$ and we compute the minimum of  
$$
h_{\bar{c}}(\beta)=\left(\frac{p-1}{q-1}+1\right)\beta^p-\bar{c}\beta+L_+,  
$$
which is attained at
\begin{equation}\label{barbeta}
\bar{\beta} = \left(\frac{\bar{c} (q-1)}{p(p+q-2)}\right)^{\frac{1}{p-1}} = 1.
\end{equation}
In view of the position done, we have $h_{\bar{c}}(\bar{\beta}) \leq 0$, where the equality holds if and only if $L_+=(p-1)(p+q-2)/(q-1)$. Since $\bar{\beta} = 1$, we have that 
$$
\mathcal{G}_{\bar{c}}(\bar{\beta}) = h_{\bar{c}}(\bar{\beta}) \leq 0, 
$$ 
yielding the conclusion. 
\item[3)]\underline{Case \eqref{c+}$_{(iii)}$: $L_+ \geq \frac{p-1}{q-1}(p+q-2)$.} 
The picture is here similar to the previous item, apart from the fact that, fixing  
$$
\bar{c}=\frac{L_+^{\frac{p-1}{p}} \, p}{(q-1)^{\frac{1}{p}} (p-1)^{\frac{p-1}{p}}} \; (p+q-2)^{\frac{1}{p}},
$$
it turns out that $\bar{\beta}$ defined in \eqref{barbeta} is greater than one for $L_+ > \frac{p-1}{q-1} (p+q-2)$, and $\mathcal{G}_{\bar{c}}(\bar{\beta}) \leq h_{\bar{c}}(\bar{\beta})=0$. 
\end{itemize}
\end{proof}
Theorem \ref{mainT} now follows combining the two estimates given in Lemma \ref{upperbd} and Proposition \ref{teo1}. 
\smallbreak
\noindent
Let us observe that the estimate provided in Proposition \ref{teo1} may be not optimal: a gap in the optimality is for sure created whenever the sign of $\mathcal{G}_c$ is different from the one of $g_c$ (or $h_c$) and the minimum of $\mathcal{G}_c$ is near $1$. A numerical search for the exact value of $\beta$ would provide better bounds, but since we are not able to solve inequality \eqref{3} without passing through a stronger control involving only one power of $\beta$, we have preferred to state the result with an analytical proof. We postpone to the next section some remarks and corollaries of Theorem \ref{mainT}.

\subsection{Some remarks and corollaries} \label{osservazioni}

We begin the section by observing that Proposition \ref{teo1} is constructed through a \emph{sufficient} condition for $c$ to be admissible (implying the existence of a positive lower solution), so that it may be that the provided bounds can be further improved with another strategy of proof. However, in the expression of $c_+$ given in Theorem \ref{mainT}, it appears a (small) range of values for $L_+$, given in \eqref{c+}$_{(ii)}$, where the process is really undergoing the interaction of the two powers $p$ and $q$ at the same level (while we see that, in cases \eqref{c+}$_{(i)}$ and \eqref{c+}$_{(iii)}$, the bounds for $c^*$ involve in a very marginal way $p$ and $q$, respectively). 
This range actually appears to exist and is not a simple drawback of our estimates: in Figure \ref{figbound} we plot the graph of the function $\mathcal{G}_c(\beta)$ for $p=4$, $q=3$, $L_+=6$ (such that $p+q-2 < L_+ \leq \frac{p-1}{q-1} (p+q-2)$) and $c$ given by \eqref{c+}$_{(i)}$, and we see that in this case $\mathcal{G}_c$ is always positive, meaning that the bound for $c_+$ given by \eqref{c+}$_{(i)}$ has to be corrected.
\begin{figure}[!h]
\centering
\includegraphics[scale=0.8]{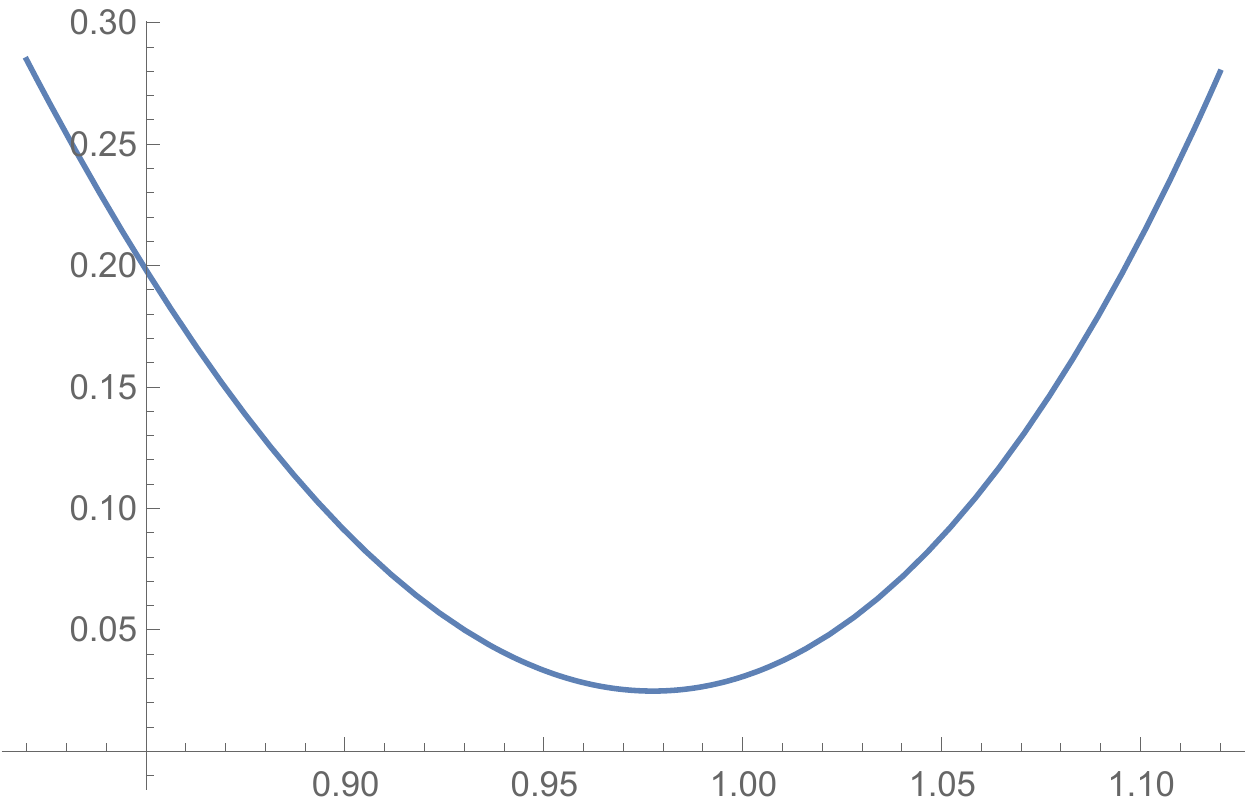}
\caption{We plot the graph of $\mathcal{G}_c(\beta)$ for the choices $p=4$, $q=3$, $L_+=6$ and $c$ given by \eqref{c+}$_{(i)}$.}
\label{figbound}
\end{figure}
\newline
In fact, when $p=q$, the range in \eqref{c+}$_{(ii)}$ disappears, while the expressions of $c_+$ in \eqref{c+}$_{(i)}$ and \eqref{c+}$_{(iii)}$ coincide. Observing that, in this case, $Q(s)=2\tfrac{q-1}{q} \vert s \vert^q$, and thus \eqref{BdR} becomes  
$$
\lim_{s\to 0} \frac{R(s)}{s^{1/q}}= \left(\frac{q}{2(q-1)}\right)^{\frac{1}{q}} \quad {\rm and} \quad \lim_{s\to \infty} \frac{R(s)}{s^{1/q}}= \left(\frac{q}{2(q-1)}\right)^{\frac{1}{q}},
$$
with the same proofs as in the previous section we are able to recover the results proved for the $p$-Laplacian, up to a constant which comes from the fact that we are considering the operator given by $(2 \vert u' \vert^{p-2} u' )'$. The following result can be compared with \cite[Proposition 5.5]{CoeSan14}.
\begin{corollary}\label{coroll}
Let $f$ be as in Theorem \ref{mainT} and $p=q$. Then, there exists $c^*$ with 
\begin{equation*}
2^{\frac{1}{q}} L_0^{\frac{1}{q'}} q'^{\frac{1}{q'}} q^{\frac{1}{q}} \leq c^* \leq 
2^{\frac{1}{q}}L_+^{\frac{1}{q'}}q^{\frac{1}{q}} {q'}^{\frac{1}{q'}}
\end{equation*}  
such that every $c \geq c^*$ is admissible for problem \eqref{IIordE}. 
\end{corollary} 
Incidentally, notice that in the particular case when $L_0=L_+$ this statement actually gives the exact value of the critical speed. 
\smallbreak
\noindent
\smallbreak
\noindent
We continue with a series of remarks. 
\begin{remark}
In the special case $p=2q$, it is possible to perform some more explicit computations. Indeed, 
since the expression for $Q(y)$ reads as
$$
Q(y)= \frac{2q-1}{2q}  y ^{2q} + \frac{q-1}{q} y^q,
$$
a straightforward computation leads to 
\begin{equation*}
R(y):=\sqrt[q]{\frac{q}{2q-1}} \sqrt[q]{-\frac{(q-1)}{q}+\sqrt{\frac{(q-1)^2}{q^2} +\frac{2(2q-1)}{q}y}}.
\end{equation*}
We notice that, on one hand, $R(y)$ can be rewritten as
\begin{equation*}
R(y)=\sqrt[q]{\frac{q}{2q-1}} \sqrt[q]{\frac{\frac{2(2q-1)}{q}y}{\frac{(q-1)}{q}+\sqrt{\frac{(q-1)^2}{q^2} +\frac{2(2q-1)}{q}y}}},
\end{equation*}
showing that, for $y \to 0$, the behavior of $R(y)$ is dictated by $y^{1/q}$; on the other hand,
\begin{equation*}
R(y)=\sqrt[q]{\frac{q}{2q-1}} y^{\frac{1}{2q}}\sqrt[q]{-\frac{(q-1)}{q \sqrt{y}}+\sqrt{\frac{(q-1)^2}{q^2 \, y} +\frac{2(2q-1)}{q}}},
\end{equation*}
implying that $R(y)$ goes to infinity like $y^{1/2q}$ for $y \to \infty$. Hence, \eqref{BdR} is satisfied, as expected. 
\end{remark} 
\begin{remark}
Theorem \ref{mainT} is extendable to the case of density-dependent diffusions, namely for an equation like
$$ 
u_t = \partial_x (\vert \partial_x D(u) \vert^{p-2} \partial_x D(u) + \vert \partial_x D(u) \vert^{q-2} \partial_x D(u)) + f(u),
$$
where $D(s)$ is a function with strictly positive derivative. The details of the change of variables in this case can be found, for instance, in
\cite[Section 2]{GarStrPP}. For the sake of briefness, we avoid giving the explicit statement, which is actually a natural adaptation of the one above. Furthermore, also the case with convection is likely to be dealt with using a similar strategy as in \cite{GarStrPP}. 
\end{remark}
\begin{remark}\label{conilmeno}
We quote another case which we find worth briefly showing, that is, the situation when the considered equation reads as 
$$
u_t=(|u_x|^{q-2} u_x- |u_x|^{p-2} u_x)_x + f(u), 
$$
namely the two terms appearing in the expression of the operator driving the diffusion are in competition. Clearly, here $p$ has to be different from $q$ and, as before, we assume $p > q$ in order for the dynamics to be comparable with the one already examined at least at infinity: namely, we want $Q(\phi(v))$ to be positive at least in the region where $\phi(v)$ is small, that is, $v \approx 0$ and $v \approx 1$. Here the operator degenerates for $\vert u_x \vert = 1$ and a problem of invertibility arises for the function $Q(s)=
\frac{q-1}{q} \vert s \vert^q - \frac{p-1}{p} \vert s \vert^p$; nevertheless, if we are able to restrict the dynamics in the neighborhood $\mathcal{I}$ of $0$ where $Q$ is invertible, then we can still recover the existence of fronts. As it is expected because of the previous discussion, in this case the critical speed $c^*$ will depend only on the power $q$; indeed, we can prove the following result.
\begin{proposition}
Let $f$ be as in Theorem \ref{mainT}; then there exists $c^*$, satisfying
\begin{equation*}
L_0^{1/q'} q'^{1/q'} q^{1/q} \leq c^* \leq L_+^{1/q'} q'^{1/q'} q^{1/q},
\end{equation*}
such that every 
\begin{equation}\label{boundc-}
c \in \left[c^*, q\left(\frac{q-1}{p-1}\right)^{\frac{q-1}{p-q}}\right]
\end{equation}
is admissible (provided that such an interval is non-empty). 
\end{proposition}
Of course, the non-emptiness of the interval appearing in \eqref{boundc-} turns out to impose a condition on $L_+$, which however appears cumbersome and not particularly significant (and for this reason we omit it). Observe also that the statement yields again the exact value of the ``critical speed'' whenever $L_0=L_+$.
\begin{proof}
Let us begin the proof by observing that the neighborhood $\mathcal{I}$ of $0$ where $Q$ is invertible is given by 
$$
\mathcal{I}= \left[0, \left(\frac{q-1}{p-1}\right)^{\frac{1}{p-q}}=:s_0\right].
$$
Hence, in searching a subsolution of the form $Q(\beta u^{\alpha})$, we have to ensure that $\beta u^\alpha \leq s_0$ for every $u \in [0, 1]$, namely $\beta \leq s_0$.  
\\
On one hand, the lower bound \eqref{bound1} for the critical speed still holds without any change. On the other hand, when trying to prove an upper bound analogous to the one given in \eqref{stima}-\eqref{c+}, by using the same strategy and assuming that $Q$ is always invertible along $\beta u^\alpha$ ($u \in [0, 1]$), we end up with
\begin{equation}\label{1bb}
(q'-1)\left[(q-1)\beta^q -(p-1)\beta^p u^{p(q'-1)-q'} \right]-c \beta  + L_+ \leq 0.
\end{equation}
Since in this case we have a negative sign in front of the term coming from the $p$-Laplacian operator and since
$
\inf_{u \in [0,1]}u^{p(q'-1)-q'} =0,
$
it turns out that it is sufficient to find a value $\bar{\beta}$ such that
$
\bar{\beta}^q -c \bar{\beta}  + L_+ \leq 0.
$
Exploiting the same computations as in the proof of Proposition \ref{teo1} leads to
\begin{equation*}
c\geq L_+^{1/q'} q'^{1/q'} q^{1/q}, 
\end{equation*}
while the upper bound in \eqref{boundc-} is needed in order for $\bar{\beta}$ to be less than $s_0$ (and thus for the expression \eqref{1bb} to be allowed). 
\end{proof}
\end{remark}

\section{Some numerical simulations}\label{sez3}

In this section, we collect a series of numerical simulations in order to better understand the role of the $p$-Laplacian in the dynamics of the traveling waves. 
\smallbreak
\noindent
We have seen in Proposition \ref{teo1} that, in spite of the fact that the operator ruling the dynamics at infinity is the $q$-Laplacian one (since the monotone wave profiles at infinity possess very small derivative), the $p$-Laplacian seems to have some influence on the value of the critical speed, especially if the reaction term is sufficiently large. 
Actually, this seems confirmed also by the numerical simulations we are going to present and is better understood by recalling a simple invariance-upon-rescaling property of the arguments presented in Section \ref{sez2}. Namely, if we assume that $f: [0, H] \to \mathbb{R}$ is such that $f(0)=0=f(H)$ and $f > 0$ on $\,]0, H[\,$, for $H > 0$ fixed, and we search for monotone connections between $0$ and $H$, we are led to solve
\begin{equation}\label{probH}
\left\{
\begin{array}{l}
\displaystyle y'(u)= c R(y(u)) - f(u), \\
\\
y(0)=0=y(H). \;
\end{array}
\right.
\end{equation}
Rescaling the independent variable by setting $v=u/H \in [0, 1]$,  
we can bring back \eqref{probH} to a problem on $[0, 1]$:
indeed, $y_c(u)$ is a solution of \eqref{probH} if and only if $w(v)=y_c(Hv)$ solves 
\begin{equation}\label{prob1}
\left\{
\begin{array}{l}
\displaystyle w'(v) = H c R(w(v)) - H f(Hv), \\
\\
w(0)=0=w(1). \;
\end{array}
\right.
\end{equation}
In the particular case of a $q$-Laplacian equation, leading to $Q(s)= \tfrac{q-1}{q} \vert s \vert^q$, so that $R(s)= (\tfrac{q}{q-1} \vert z \vert)^{1/q}$, the assumption on the reaction term guaranteeing the exact computation of the critical speed reads as  
$$
\sup_{u \in \,]0, H]} \frac{f(u)}{u^{q'-1}}=\lim_{u \to 0} \frac{f(u)}{u^{q'-1}} = L,
$$
which turns out into asking that, setting $g(v)= Hf(Hv)$, it holds
$$
M:=\sup_{v \in \,]0, 1]} \frac{g(v)}{v^{q'-1}}=\lim_{u \to 0} \frac{g(v)}{v^{q'-1}} = H^{q'} L. 
$$
However, from the computation of the critical speed done in \cite{CoeSan14} (see also Corollary \ref{coroll}) in this particular case, we have that \eqref{prob1} has a solution if and only if 
$$
Hc \geq M^{\frac{1}{q'}} q'^{\frac{1}{q'}} q^{\frac{1}{q}},
$$
but this is equivalent to  
\begin{equation}\label{bound3}
c \geq L^{\frac{1}{q'}}q'^{\frac{1}{q'}} q^{\frac{1}{q}}, 
\end{equation}
namely it is possible to compute the critical speed for \eqref{probH} with the same formula as for \eqref{prob1}. 
\smallbreak
\noindent
Incidentally, we notice that assuming instead $f(u) \leq f'(0) u$ for every $u \in [0, H]$, the same argument allows to infer that the critical speed for problem \eqref{probH} with $R(s)=\sqrt{2s}$ (namely, for the classical Fisher equation) is always given by $c^*=2\sqrt{f'(0)}$.  
\smallbreak
\noindent
This argument fails for our equation, as we expect also in view of the discussion at the beginning of Section \ref{osservazioni}. We show indeed two numerical simulations (Figures \ref{A} and \ref{B}) with a completely different behavior: in the first one, for a small $H$, we show that the critical speed appears to be the one associated with the $q$-Laplacian, given by the bound \eqref{bound3}; the second one, for a larger $H$, shows that the critical speed is strictly larger than the one provided by such bound.  
Again, we remark that this seems in line with the intuitive observation that, the more there is room for $v'$ to possibly become large (i.e., the larger is $H$), the more the $p$-Laplacian passes to govern the diffusion process. Indeed, in Figure \ref{A} we have that $y(u)$ is always very small, while this is not the case for Figure \ref{B}. We refer to the captions for further comments. 
\bigbreak
\bigbreak
\begin{figure}[!h]
\centering
\includegraphics[scale=0.8]{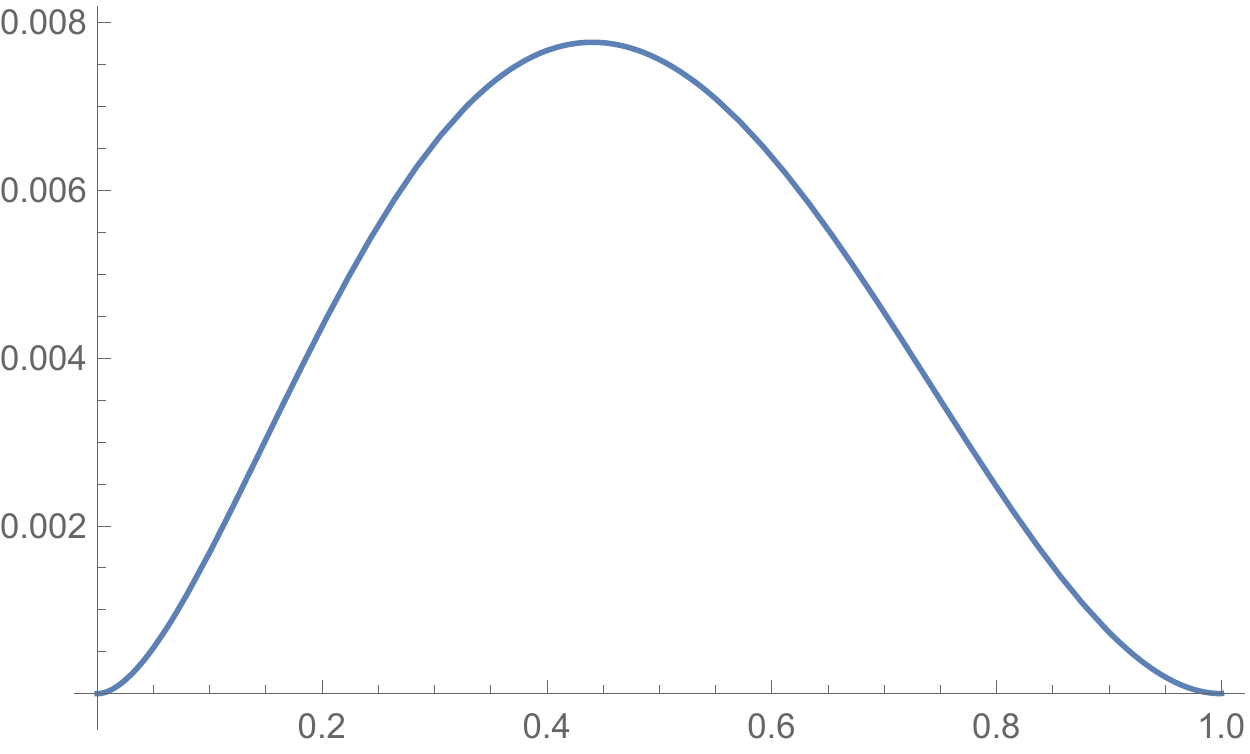}
\caption{We plot the solution of \eqref{probH} with $p=4$, $q=2$, $H=1$, and $f(u)= u^{q'-1} (1-u)$, for $c=2$. Actually, with the notations used throughout the section, here $L=1$, $q=q'=2$, so the lower bound for $c^*$ is equal to $2$ and it actually seems that $c=2$ is already an admissible speed. Notice that indeed $y'(u)$ appears very small, so the process appears ruled by the $q$-power appearing in the expression of the operator. }
\label{A}
\end{figure}

\begin{figure}[!h]
\centering
\includegraphics[scale=0.8]{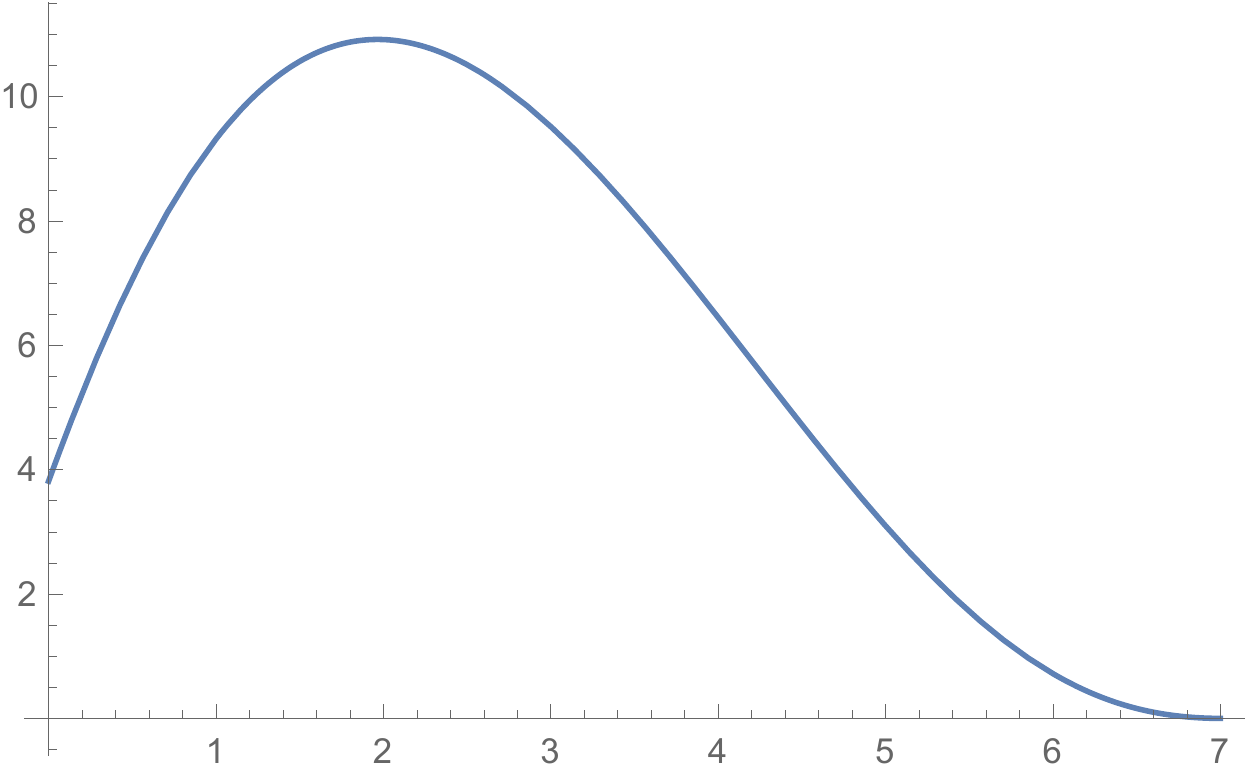}
\caption{We plot the solution of \eqref{cback} with $p=4$, $q=2$, $H=7$, $f(u)= u^{q'-1} (H-u)$ and $c=2\sqrt{7}$; here $L=7$ and the bound \eqref{bound3} would give $c^* \geq 2\sqrt{7}$. However, this speed appears far from being admissible and, indeed, $y$ reaches values much greater than $1$, corresponding to regions where the $p$-Laplacian rules the diffusive process.}
\label{B}
\end{figure}

The third simulation we show is instead devoted to a competitive case, namely to the situation described in Remark \ref{conilmeno}, where the operator is given by the difference between a $q$-Laplacian and a $p$-Laplacian. As already remarked, in this case a problem of invertibility of the function $Q(s)$ arises, and we can search for monotone connections only as long as we remain in the region where its inverse $R(\cdot)$ is well-defined. In Figure \ref{C}, we show a case where the interval appearing in \eqref{boundc-} is not empty and the critical speed is computed exactly; in the situation depicted in Figure \ref{D}, instead, there are no admissible speeds (again we refer the reader to the captions for further comments).

\begin{figure}[!h]
\centering
\includegraphics[scale=0.8]{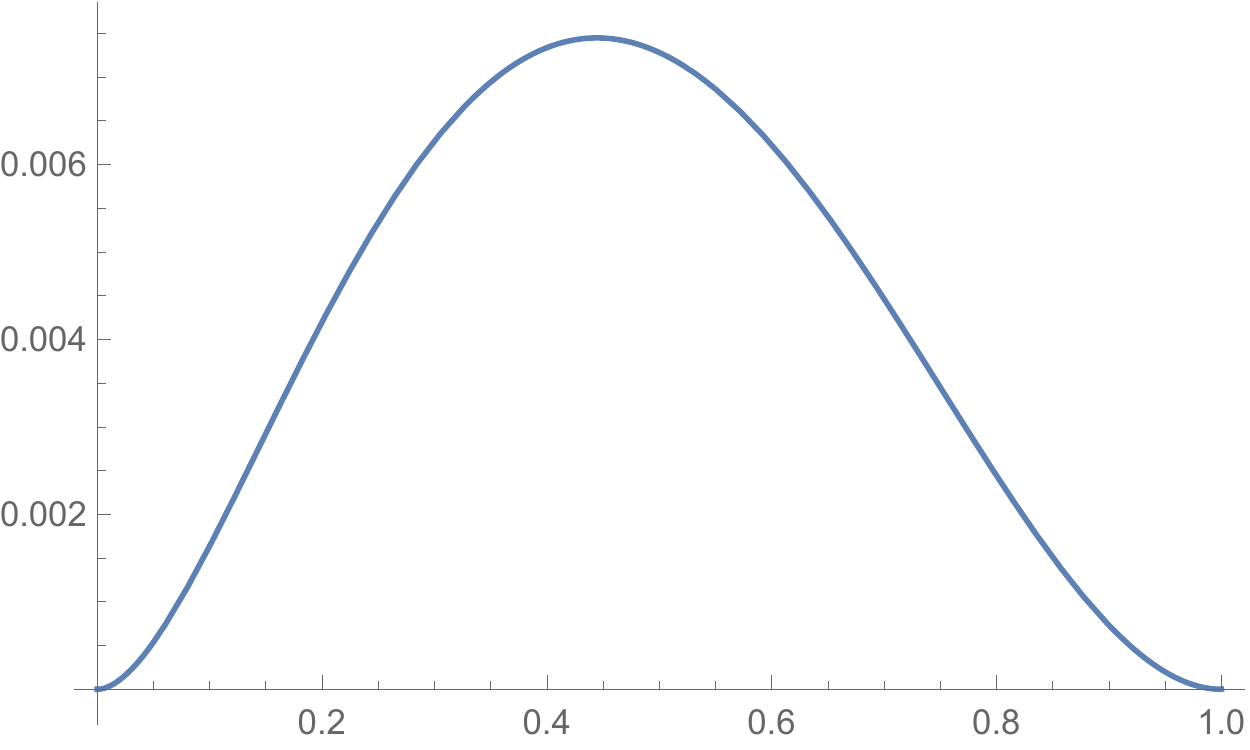}
\caption{We plot the solution of \eqref{probH} with $p=4$, $q=2$, $Q(s)=\tfrac{q-1}{q} \vert s \vert^q - \tfrac{p-1}{p} \vert s \vert^p$, $f(u)= u^{q'-1} (1-u)$ and $c=2$; here $y(u)$ stays well below the threshold of invertibility of $Q$, so that the situation roughly appears the same as in Figure \ref{A} and the bound given by \eqref{bound1} already gives an admissible speed.}
\label{C}
\end{figure}

\clearpage

\begin{figure}[!h]
\centering
\includegraphics[scale=0.8]{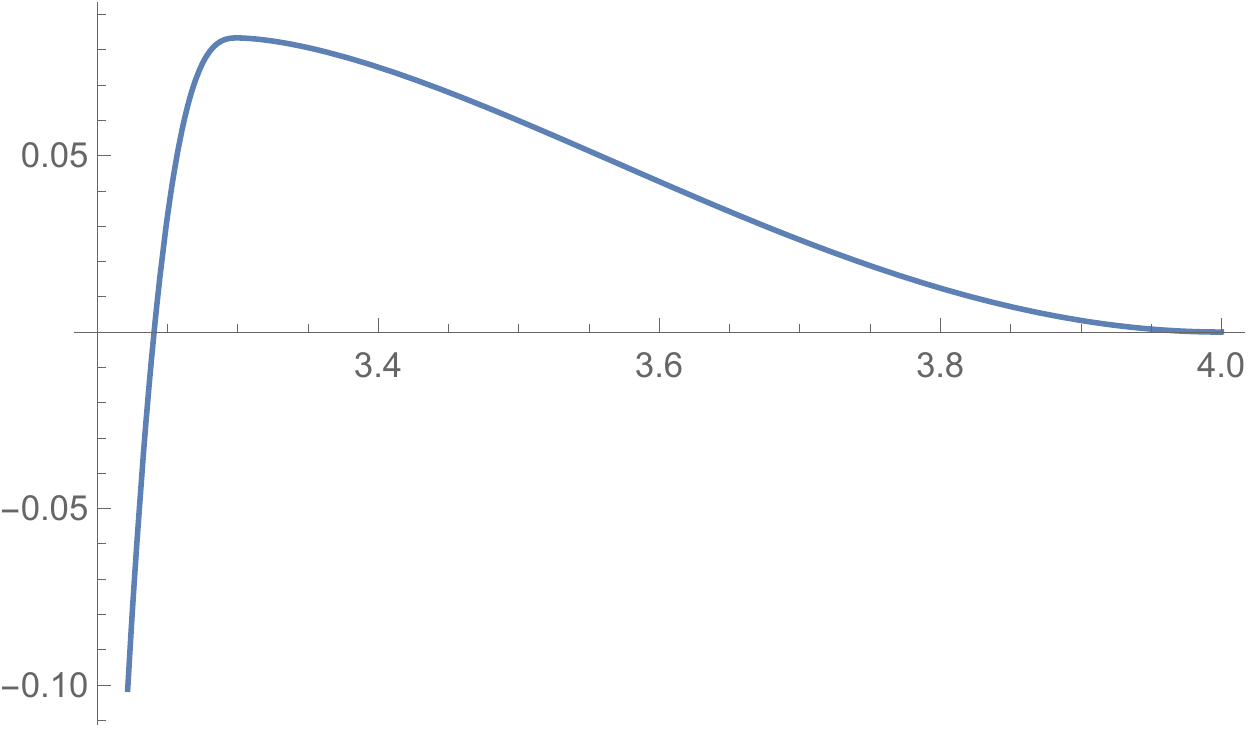}
\caption{On the other hand, for $f(u)= u^{q'-1} (4-u)$ (and the other positions as in Figure \ref{C}) the solution of \eqref{cback} has to increase too much its derivative in order to connect the two equilibria $0$ and $4$, and this is reflected in the fact that we end up in the non-invertibility region for $Q$, where the problem loses its well-posedness. Thus, we are not able to find any solution of \eqref{generale}. Again, we stress that the inverse $R(\cdot)$ has to be defined at least near $0$, since at infinity the slope of the wave profile is almost horizontal. }
\label{D}
\end{figure}

\end{document}